\documentclass{amsart}

\newtheorem{theorem}{Theorem}[section]
\newtheorem{lemma}[theorem]{Lemma}

\theoremstyle{definition}

\theoremstyle{corollary}

\theoremstyle{remark}
\newtheorem{remark}[theorem]{Remark}

\numberwithin{equation}{section}
\begin{document}

\title{Global well-posedness for the 2D Boussinesq Equations with Zero Viscosity}

\author{Daoguo Zhou}
\address{College of Mathematics and Informatics, Henan Polytechnic University, Jiaozuo, Henan 454000, China}
\email{daoguozhou@hpu.edu.cn}

\author{Zilai Li}
\address{College of Mathematics and Informatics, Henan Polytechnic University, Jiaozuo, Henan 454000, China}
\email{lizl@hpu.edu.cn}

%    General info


\begin{abstract}
We prove the global well-posedness of the two-dimensional Boussinesq equations
with zero viscosity and positive diffusivity in bounded domains for rough initial data [ $u_{0}\in L^{2}$, $\text{curl}\,u_{0}\in L^{\infty}$ and $\theta_{0}\in B^{2-2/p}_{q,p}$ with $p\in (1,\infty)$, $q\in (2,\infty)$ ]. Our method is based on the maximal regularity for heat equation.
\end{abstract}

\maketitle

\section{Introduction}
In the present paper, we investigate  the global  well-posedness of the 2D Boussinesq equations with zero viscosity and positive diffusivity in a smooth bounded domain $\Omega\subset \mathbb{R}^2$ with smooth boundary $\partial \Omega$. The corresponding system reads
\begin{equation}\label{BOUo}
\begin{cases}
\partial_t u+ u\cdot\nabla u +\nabla p =  \theta e_2,   \\
\partial_t \theta+ u\cdot\nabla \theta-\kappa\Delta \theta = 0,   \\
\nabla \cdot u =0,
\end{cases}
\end{equation}
where $u$ is the velocity vector field, $p$ is the pressure,  $\theta$ is the temperature, $\kappa>0$ is the thermal diffusivity, and $e_2=(0,1)$. We supplement the system (1.1) with the following initial boundary value conditions
\begin{equation}\label{IBo}
\begin{cases}
(u,\theta)(x,0)=(u_{0},\theta_{0})(x),\;x\in\Omega,\\
u(x,t)\cdot n|_{\partial\Omega}=0, \theta(x,t)|_{\partial\Omega}=\bar \theta,
\end{cases}
\end{equation}
where $n$ is the outward unit normal vector to $\partial \Omega$, and $\bar \theta$ is a constant.

The general 2D Boussinesq equations with viscosity $\nu$ and diffusivity $\kappa$ are
\begin{equation*}\label{GBOU}
\begin{cases}
\partial_t u+ u\cdot\nabla u -\nu\Delta u+\nabla p  =  \theta,   \\
\partial_t \theta+ u\cdot\nabla \theta-\kappa\Delta \theta = 0,   \\
\nabla \cdot u =0.
\end{cases}
\end{equation*}
The Boussinesq equations are of relevance to study a number of models coming from
atmospheric or oceanographic turbulence where rotation and stratification play an important role (see \cite{Ma}, \cite{Pe}). From the mathematical view, the
2D Boussinesq equations serve as a simplified model of the 3D Euler and Navier-
Stokes equations (see \cite{Ma2}). Better understanding of the 2D Boussinesq equations will shed
light on the understanding of 3D flows.

Recently, there are many works devoted to the well-posedness of the 2D Boussinesq equations, see \cite{Abidi}-\cite{Huang}, \cite{Xu}-\cite{Zhou}. In particular, when $\Omega=\mathbb{R}^2$, Chae in \cite{Chae} showed that the
system \eqref{BOUo}-\eqref{IBo} has a  global smooth solution for $(u_0,\theta_0)\in H^3$. In the bounded domains case,  the boundary effect requires a careful mathematical analysis. In this direction, Zhao in \cite{Zhao} was able to generalize the study of \cite{Chae} to smooth bounded domains. This result was later extended by Huang in \cite{Huang} to the case of Yudovich's type data: $\text{curl}\,u_0\in L^\infty$ and $\theta_0\in H^2$. We intend here to improve Huang's result further by lowering the regularity for initial data.

Our main result is stated in the following theorem.
\begin{theorem}\label{Th}
Let $\Omega$ be a bounded domain in $\mathbb{R}^2$ with $C^{2+\epsilon}$ boundary for some $\epsilon>0$. Suppose that $u_0\in L^2$, $\text{curl}\, u_0\in L^\infty$, and $\theta_0\in B^{2-2/p}_{q,p}$ with $p\in(1,\infty)$, $q\in(2,\infty)$. Then there exists a unique global solution $(u,\theta)$ to the system \eqref{BOU}-\eqref{IB}, which satisfies that for all $T>0$
\begin{gather*}
\theta\in C([0,T];B_{q,p}^{2-2/p})\cap L^{p}(0,T;W^{2,q})\,,
 \partial_t\theta \in  L^{p}(0,T;L^{q}),\\
u \in L^\infty(0,T;L^2) \text{ and }\text{curl}\, u\in L^\infty(0,T;L^\infty).
\end{gather*}
\end{theorem}
\begin{remark}
We only require mild regularity for the initial temperature
$\theta_0$, as the ``regularity index'' $2-2/p$ can be arbitrarily
closed to zero. Thus, our result significantly improves the previous results  \cite{Chae,Huang,Zhao}.
\end{remark}
\begin{remark}
By modifying slightly the method in the current paper, we can prove the global well-posedness for initial data $(u_0,\theta_0)\in H^{2+s}\times H^s$ with $s>0$ or  $(u_0,\theta_0)\in W^{2,q}\times W^{1,q}$ with $q>2$.
\end{remark}

The proof of Theorem \ref{Th} consists of two main steps. First, we show the global
existence of weak solutions to \eqref{BOUo}-\eqref{IBo}. Then we improve the regularity of weak solutions using the maximal regularity  for heat equation. Our proof is elementary and can be carried over to the case of $\mathbb{R}^2$ without difficulty.

The rest of our paper is organized as follows. In Section 2, we recall maximal regularity for heat equations as well as some basic facts.  Section 3 is devoted to the proof of our main theorem.
\section{Notations and Preliminaries}
\noindent{\bf
Notations:}

(1)Let $\Omega$ be a bounded domain in $\mathbb{R}^{2}$. For $p\ge 1$ and  $k\ge 1$, $L^p(\Omega)$ and $W^{k,p}(\Omega)$ ($p=2$, $H^k(\Omega)$)
denote the standard Lesbegue space and Sobolev space respectively. For $T>0$ and
a function space $X$, denote by $L^p(0,T; X)$ the set of Bochner measurable $X$-valued time dependent
functions $f$ such that $t\to \|f\|_X$ belongs to $L^p(0,T)$.

(2)Let $s\in (0,\infty)$, $p\in (1, \infty)$ and $r\in [1, \infty]$. The Besov space $B_{p,r}^{s}(\Omega)$ is defined as the real interpolation space between $L^p(\Omega)$ and $W^{m,p}(\Omega)$ ($m>s$),
\[B_{p,r}^{s}(\Omega)=(L^p(\Omega),W^{m,p}(\Omega))_{\frac s m,r}.\]
See Adams and Fournier (\cite{AD}, Chapter 7).

(3)Throughout this paper, the same letter $C$ denotes various generic positive constant which is dependent on initial data $(u_0,\theta_0)$, time $T$, the thermal diffusivity  $\kappa$, and the domain $\Omega$.

We need the well-known Sobolev Embeddings and Gagliardo-Nirenberg inequality (see Adams and Fournier \cite{AD} and  Nirenberg \cite{Niren}).

\begin{lemma}\label{embed}
	Let $\Omega\in \mathbb{R}^{2}$ be any bounded domain with $C^2$ boundary. Then
	the following embeddings and inequalities hold:
	\begin{enumerate}
		\item $ H^1(\Omega)\hookrightarrow L^q(\Omega)$, for all $q\in (1,\infty)$.
		\item $\|\nabla u\|_{L^\infty}\leq C \|\nabla^2u\|_{L^q}^{\alpha}\|u\|_{L^2}^{1-\alpha}+C\|u\|_{L^2}$, for all $u\in W^{2,q}(\Omega)$, with $q\in (2,\infty)$, $\alpha=\frac {2q} {3q-2}$ and $C$ is a constant depending on $q,\Omega$.
	\end{enumerate}
\end{lemma}

We also need the maximal regularity for heat equation (see Amann \cite{Amann}),
which is critical to the proof of our main theorem.
\begin{lemma}\label{le:2.2}
Let $\Omega$ be a bounded domain with a $C^{2+\epsilon}$ boundary in $\mathbb{R}^2$ and $1 < p, q <\infty$.
Assume that $u_0\in B_{q,p}^{2-2/p}$, $f \in L^p(0,\infty;L^q)$. Then the system
\begin{equation}
\begin{cases}
\partial_t u-\kappa\Delta u=f,\\
u(x,t)|_{\partial \Omega}=0,\\
u(x,t)|_{t=0}=u_0,
\end{cases}
\end{equation}
has a unique solution $u$ satisfying the following inequality for all $T > 0$:
\begin{align*}
&\|u(T)\|_{B_{q,p}^{2-2/p}}+\|u\|_{L^p(0,T; W^{2,q})}+\|\partial_t u\|_{L^p(0,T; L^q)}\\
&\leq C\left(\|u_0\|_{B_{q,p}^{2-2/p}}+\|f\|_{L^p(0,T; L^q)}\right),
\end{align*}
with $C=C(p,q,\kappa,\Omega)$.
\end{lemma}
We complete this section by recalling the following well-known inequality (see Yudovich \cite{Yudovich}), which will be used several times.
\begin{lemma}\label{le:2.4}
For any $p\in (1,\infty)$, the following estimate holds
\[\|\nabla u\|_{L^p(\Omega)}\leq C \frac{p^2}{p-1}\|w\|_{L^p(\Omega)},\]
where $C=C(\Omega)$ does not depend on $p$.
\end{lemma}

\section{Proof of Main Theorem}

We first reformulate the initial-boundary value problem \eqref{BOUo}-\eqref{IBo}. Let $\bar p= p-\theta y$ and $\Theta=\theta-\bar \theta $, then we get from the original system
\begin{equation}\label{BOU}
\begin{cases}
\partial_t u+ u\cdot\nabla u +\nabla \bar p =  \Theta e_2,   \\
\partial_t \Theta+ u\cdot\nabla \Theta-\kappa\Delta \Theta = 0,   \\
\nabla \cdot u =0,
\end{cases}
\end{equation}
The initial and boundary conditions become
\begin{equation}\label{IB}
\begin{cases}
(u,\Theta)(x,0)=(u_{0},\Theta_{0})(x),\;x\in\Omega,\\
u(x,t)\cdot n|_{\partial\Omega}=0, \Theta(x,t)|_{\partial\Omega}=0,
\end{cases}
\end{equation}
where $\Theta_0=\theta_0-\bar \theta$. It is clear that \eqref{BOU}-\eqref{IB} are equivalent
to \eqref{BOUo}-\eqref{IBo}. Hence, for
the rest of this paper, we shall work on the reformulated problem \eqref{BOU}-\eqref{IB}.
\subsection{Existence}
First, we show the existence of weak solutions. Then, we improve the regularity using the maximal regularity for heat equation. In fact, one can establish the global existence of weak solution to \eqref{BOU}-\eqref{IB} in bounded domains
by a standard argument, see Zhao \cite{Zhao}.
\begin{lemma}\label{weak}
	Let $\Omega$ be a bounded domain in $\mathbb{R}^2$. Assume that $(u_0,\Theta_0)\in L^2\times L^2$. Then
	there exists one solution $(u,\Theta)$ to \eqref{BOU}-\eqref{IB} such that for any $T>0$
	\begin{enumerate}
		\item
		$u\in L^\infty (0,T;L^2)$, $\Theta\in L^\infty (0,T;L^2)\cap L^2(0,T;H^1)$.
		\item
		$\int_{\Omega}u_0\phi(x,0)dx+\int_0^T\int_{\Omega}\left(u\cdot \partial_t\phi+u\cdot\left(u\cdot \nabla \phi\right)+\Theta e_2\phi \right)dxdt=0$, for any vector function $\phi\in C_0^\infty(\Omega\times[0,T))$ satisfying
		$\nabla \cdot \psi=0$.
		\item
		$\int_{\Omega}\Theta_0\psi(x,0)dx+\int_0^T\int_{\Omega}\left(\Theta\cdot \partial_t\psi+\Theta u\cdot \nabla \psi-\nabla \Theta\cdot\nabla \psi\right)dxdt=0$, for any scalar function $\psi\in C_0^\infty(\Omega\times[0,T))$.
	\end{enumerate}
\end{lemma}
It remains to establish the global regularity of solutions obtained in Lemma \ref{weak}. The proof is divided into several lemmas.
\begin{lemma}\label{le:3.1}
Let the assumptions in Theorem \ref{Th} hold.
Then the solution obtained in Lemma \ref{weak} satisfies
\begin{equation}\label{le:3.1-1}
\Theta\in L^\infty(0,T;L^2)\cap L^2(0,T;H^1),
\end{equation}
\begin{equation}\label{le:3.1-2}
u\in L^\infty(0,T;L^2).
\end{equation}
\end{lemma}
\begin{proof} Multiplying \eqref{BOU}$_{2}$ by $T$ and  integrating it over $\Omega$ by parts, we find
\[\|\Theta\|_{L^{2}}^{2}+2\kappa\int_{0}^{T}\|\nabla \Theta\|_{L^{2}}^{2}dt\leq\|\Theta_{0}\|_{L^2}^{2}.\]

For second estimate, taking $L^2$ inner product
of \eqref{BOU}$_{1}$ with $u$, and  using H\"older's inequality, we get
\[\frac{1}{2}\frac{d}{dt}\|u\|_{L^2}^2=\int_{\Omega}\Theta e_2 \cdot udx\leq \|\Theta\|_{L^2}\|u\|_{L^2},\]
which, by the Cauchy-Schwarz inequality, gives that
\[\|u\|_{L^2}\leq \|u_0\|_{L^2}+\int_0^T\|\Theta\|_{L^2}ds\leq \|u_0\|_{L^2}+T\|\Theta_0\|_{L^2}.\]
Then the proof of Lemma \ref{le:3.1} is finished.
\end{proof}
\begin{lemma}\label{le:3.2}
Let the assumptions in Theorem \ref{Th} hold. Then
the solution obtained in Lemma \ref{weak} satisfies
\begin{equation}\label{le:3.2-2}
w\in L^\infty(0,T;L^2).
\end{equation}
\end{lemma}
\begin{proof} We recall that the vorticity $w=\text{curl}\; u$ satisfies the equation
\begin{equation}\label{3.2-1}
\partial_t w+u\cdot\nabla w=\partial _1 \Theta.
\end{equation}
Multiplying \eqref{3.2-1} by $w$, integrating the resulting equations over $\Omega$ by parts, and using H\"older's inequality, we have
\[\frac{1}{2}\frac{d}{dt}\|w\|_{L^2}^2=\int_{\Omega}\partial_1\Theta wdx\leq \|\nabla \Theta\|_{L^2}\|w\|_{L^2},\]
which implies that
\begin{align*}
\|w\|_{L^2}
&\leq \|w_0\|_{L^2}+\int_0^T\|\nabla \Theta\|_{L^2}ds\\
&\leq \|w_0\|_{L^2}+T^{1/2}\left(\int_0^T\|\nabla \Theta\|_{L^2}^2ds\right)^{1/2}\\
&\leq \|w_0\|_{L^2}+\frac{T^{1/2}}{\sqrt{2\kappa}}\|\Theta_0\|_{L^2}.
\end{align*}
Then the
proof of Lemma \ref{le:3.2} is finished.
\end{proof}
\begin{lemma}\label{le:3.3}
Let the assumptions in Theorem \ref{Th} hold. Then
the solution obtained in Lemma \ref{weak} satisfies
\begin{gather*}
\Theta\in C^\infty([0,T];B_{q,p}^{2-2/p})\cap L^{p}(0,T;W^{2,q})\,,
 \partial_t \Theta \in  L^{p}(0,T;L^{q}),\\
u \in L^\infty(0,T;L^2) \text{ and }\text{curl}\, u\in L^\infty(0,T;L^\infty).
\end{gather*}
\end{lemma}
\begin{proof}
First, we obtain from \eqref{le:3.2-2} and Lemma \ref{le:2.4} that
\[\nabla u\in L^\infty(0,T;L^2),\]\
which implies that for $2\leq q <\infty$,
\begin{equation}\label{3.2-2}
u\in  L^\infty(0,T;L^q).
\end{equation}

Considering the equation for the temperature, by the maximal regularity for heat equation and H\"{o}lder's inequality, we obtain that for $1<p<\infty$, $2<q<\infty$,
\begin{equation}\label{3.2-3}
\begin{aligned}
&\quad\|\Theta\|_{L^\infty(0,T;B_{q,p}^{2-2/p})}+\|\Theta\|_{L^p(0,T; W^{2,q})}+\|\partial_t \Theta\|_{L^p(0,T; L^q)}\\
&\leq C\|\Theta_0\|_{B_{q,p}^{2-2/p}}+C \|u\cdot\nabla \Theta\|_{L^p(0,T;L^q)}\\
&\leq C\|\Theta_0\|_{B_{q,p}^{2-2/p}}+C \|u\|_{L^\infty(0,T;L^q)}\|\nabla \Theta\|_{L^p(0,T;L^\infty)}\\
&\leq C\|\Theta_0\|_{B_{q,p}^{2-2/p}}+C\|\nabla \Theta\|_{L^p(0,T;L^\infty)}.
\end{aligned}
\end{equation}
Using the interpolation inequality in Lemma \ref{embed}, H\"{o}lder's inequality and Young's inequality, we have for arbitrary $\epsilon >0$, $q>2$,
\[\|\nabla \Theta\|_{L^p(0,T;L^\infty)}\leq \epsilon \|\nabla^2 \Theta\|_{L^p(0,T;L^q)}+C(\epsilon)\|\Theta\|_{L^p(0,T;L^2)}.\]
Plugging the above inequality into \eqref{3.2-3},  absorbing the small $\epsilon$ term, we get
\[\|\Theta\|_{L^\infty(0,T;B_{q,p}^{2-2/p})}+\|\Theta\|_{L^p(0,T; W^{2,q})}+\|\partial_t \Theta\|_{L^p(0,T; L^q)}\leq C(p,q,\kappa,T,\Omega,u_0,\Theta_0),\]
which yields that
\[\nabla \Theta \in L^1(0,T;L^\infty).\]

Coming back to the vorticity equation \eqref{3.2-1}, we derive that
\[w\in L^\infty(0,T;L^\infty).\]
Then the proof of Lemma \ref{le:3.3} is finished.
\end{proof}
\subsection{Uniquenss}
The method adapted here is essentiallly due to Yudovich \cite{Yudovich}, see also Danchin \cite{Danchin}.

Let $(u_1,\Theta_1,\bar p_1)$ and $(u_2, \Theta_2,\bar p_2)$ be two solutions of the system \eqref{BOU}-\eqref{IB}. Denote $\delta u=u_1-u_2$, $\delta \Theta=\Theta_1-\Theta_2$, and $\delta p=\bar p_1-\bar p_2$. Then $(
\delta u, \delta \Theta,\delta p)$ satisfy
\begin{equation}
\begin{cases}
\partial_t \delta u+ u_2\cdot\nabla \delta u +\nabla \delta p =-\delta u\cdot \nabla u_1+  \delta \Theta e_2,   \\
\partial_t\delta \Theta -\kappa \Delta \delta \Theta = -u_2\cdot\nabla\delta \Theta-\delta u\cdot\nabla \Theta_1 ,   \\
\nabla \cdot \delta u =0, \\
\delta u(x,t)\cdot n|_{\partial\Omega}=0, \delta \Theta(x,t)|_{\partial\Omega}=0,\\
\delta u(x,0) =0, \quad \delta \Theta(x,0) =0.
\end{cases}
\end{equation}
By standard energy method and H\"{o}lder's inequality, we have for all
$r\in [2,\infty)$
\begin{align*}
\frac{1}{2}\frac{d}{dt}\|\delta u\|_{L^2}^2
&\leq \|\nabla u_1\|_{L^r}\|\delta u\|_{L^{2r'}}^2+\|\delta \Theta \|_{L^2}\|\delta u\|_{L^2}\\
&\leq \|\nabla u_1\|_{L^r}\|\delta u\|_{L^{\infty}}^{2/r}\|\delta u\|_{L^2}^{2/r'}+\|\delta \Theta \|_{L^2}\|\delta u\|_{L^2},\\
\intertext{and}
\frac{1}{2}\frac{d}{dt}\|\delta \Theta\|_{L^2}^2
&\leq \|\nabla \Theta_1\|_{L^\infty}\|\delta \Theta\|_{L^2}\|\delta u\|_{L^2},
\end{align*}
where $1/r+1/r'=1$.

Denoting $X(t)=\|\delta u\|_{L^2}^2+\|\delta \Theta\|_{L^2}^2$,
we find that
\[\frac{1}{2}\frac{d}{dt}X
\leq \|\nabla u_1\|_{L^r}\|\delta u\|_{L^{\infty}}^{2/r}X^{1/r'}+\frac{1}{2}(1+\|\nabla \Theta_1\|_{L^\infty})X.\]
Setting $Y=e^{-\int_0^t\left(1+\|\nabla \Theta_1\|_{L^\infty}\right)ds}X$, we deduce that
\begin{align*}
\frac{1}{r}Y^{-\frac{1}{r'}}\frac{d}{dt}Y
&\leq 2e^{-\frac{1}{r}\int_0^t\left(1+\|\nabla \Theta_1\|_{L^\infty}\right)ds}\|\nabla u_1\|_{L^r}\|\delta u\|_{L^\infty}^{2/r}\\
&\leq \frac{2}{r}\|\nabla u_1\|_{L^r}\|\delta u\|_{L^\infty}^{2/r}
\end{align*}
Integrating in time on $[0,t]$ gives us that
\begin{equation}\label{eqY}
Y(t)\leq \left(2\int_0^t\frac{\|\nabla u_1\|_{L^r}}{r}\|\delta u\|_{L^\infty}^{2/r}ds\right)^r.
\end{equation}

To proceed, we make two simple observations. First, combing  Lemma \ref{le:2.4} and the bound $w_1\in L^\infty(0,T;L^\infty)$, we deduce that
\begin{equation}
\sup_{1\leq r<\infty}\frac{\|\nabla u_1(t)\|_{L^r}}{r}\leq C(\Omega).
\end{equation}
Second, from the fact that $u_i\in L^\infty(0,T; L^2)$ and  $w_i\in L^\infty(0,T; L^\infty)$ for $i=1,2$, we have
\begin{equation}\label{bounddu}
\delta u\in L^\infty(0,T;L^\infty).
\end{equation}

Next, choosing
$T^*$ such that $\int_0^{T^*}\frac{\|\nabla u_1\|_{L^r}}{r}\leq 1/4$, together with \eqref{bounddu}, we can rewrite \eqref{eqY} as
\[Y(t)\leq C\left(\frac{1}{2}\right)^r.\]
Sending $r$ to $\infty$, we get $Y(t)\equiv 0$
on $[0,T^*]$. By standard induction argument, we conclude that $Y(t)\equiv 0$ for all $t>0$, which means that
$(\delta u, \delta T,\delta p)\equiv 0$ for all $t\in [0,T]$. Thus we obtain the uniqueness of solutions.

 \

 \noindent{\bf\small Acknowledgment.}

{\small D.G. Zhou is supported by the National Natural Science Foundation of China (No. 11401176) and Doctor Fund of Henan Polytechnic University (No. B2012-110). Z.L. Li is supported by Doctor Fund of Henan Polytechnic University (No. B2016-57), the Fundamental Research Funds for the Universities of Henan Province (NSFRF 16A110015) and the National Natural Science Foundation of China (No. 11601128).}

\bibliographystyle{amsplain}

\end{document}